\documentclass{article}
\usepackage[utf8]{inputenc}
\usepackage{latexsym, amsmath, amsfonts, amscd, amssymb, verbatim}
\usepackage[unicode]{hyperref}
\usepackage[margin=1in]{geometry}  
\usepackage{graphicx}
\usepackage{subfigure}
\usepackage{color}
\usepackage{cite}

\newtheorem{Theorem}{Theorem}[section]
\newtheorem{Proposition}{Proposition}[section]

\newtheorem{Definition}{Definition}[section]
\newtheorem{Example}{Example}[section]
\normalsize
\newtheorem{Remark}{Remark}[section]

\newcommand {\dla} {\Delta}
\newcommand{\inner}[2]{\left\langle {#1}, {#2} \right\rangle}



\title{Second order continuous and discrete dynamical systems for solving inverse quasi-variational inequalities }
\author{ 
Pham Viet Hai\thanks{Faculty of Mathematics and Informatics, Hanoi University of Science and Technology, 1 Dai Co Viet, Hanoi, Vietnam. Email: hai.phamviet@hust.edu.vn}
\and
		 Thanh Quoc Trinh\thanks{
				Faculty of Fundamental Sciences, Van Lang University,  Ho Chi Minh City, Vietnam. Email: thanh.tq@vlu.edu.vn}
			\and
		Phan Tu Vuong\thanks{Mathematical Sciences School, University of Southampton, SO17 1BJ, Southampton, UK. Email: t.v.phan@soton.ac.uk}}
\date{}

\begin{document}

\maketitle
\begin{quote}
\noindent {\bf Abstract.} 
In this paper, we investigate the inverse quasi-variational inequality problem in finite-dimensional spaces. First, we introduce a second-order dynamical system whose trajectory converges exponentially to the solution of the inverse quasi-variational inequality, under the assumptions of Lipschitz continuity and strong monotonicity. Next, we discretize the proposed dynamical system to develop an algorithm, and prove that the iterations converge linearly to the unique solution of the inverse quasi-variational inequality. Finally, we present numerical experiments and applications to validate the theoretical results and compare the performance with existing methods.

\medskip
\noindent {\bf Mathematics Subject Classification (2010).}\ 47J20,
49J40, 49M30.

\medskip
\noindent {\bf Key Words.} Inverse quasi-variational inequality; second order dynamical system; linear convergence; traffic assignment problem.  
\end{quote}

\section{Introduction}
The variational inequality problem (VIP) is a valuable mathematical model used to represent phenomena in both theoretical and applied fields, including optimization problems, fixed-point problems, complementarity problems, and Nash equilibrium problems (see \cite{FacchineiPang03, Kinderlehrer1980, Nagurney1996}). A formulation of the problem in the $n-$dimensional Euclidean space $\mathbb{R}^n$ is the following: given a nonempty, closed, convex subset $C$ of $\mathbb{R}^n$ and a continuous mapping $F: \mathbb{R}^n \rightarrow \mathbb{R}^n$, the VIP consists of finding $x^*\in C$ such that
\[\langle F(x^*),x-x^*\rangle \geq 0, \, \forall x\in C,\]
where $\langle \cdot,\cdot\rangle$ is the standard inner product on $\mathbb{R}^n$.

In practice, there are many cases where the mapping $F$ does not appear as an explicit formula (for instance, see \cite{HeX2011}), but as the inverse of some mapping, i.e. $F=f^{-1}$. This case leads to the inverse variational inequality problem (IVIP) of finding $z^* \in \mathbb{R}^n$ subject to the condition
\[f(z^*) \in C \quad \text{and}\quad  \langle z^*, t-f(z^*)\rangle \geq 0,\quad \forall t\in C.\]

Algorithms in projection form have been developed to solve VIPs and IVIPs (see \cite{Censor2011, Korpelevich76, Popov80, Vuong2020, He2010, Vuong2021}). A key approach to designing such algorithms is the discretization of dynamical systems. The connection between optimization algorithms and dynamical systems has been recognized for a long time and is actively explored by research groups (see \cite{Dey2019, Gao2005, Nagurney1996, Nguyen2020, Pappalardo2002}). In the recent survey \cite{Csetnek2020}, Csetnek systematically reviewed recent advances in the study of first- and second-order dynamical systems in relation to monotone inclusions.

For VIPs, algorithms have been generated from both first-order (see \cite{Cavazzuti2002, Ha2018}) and second-order (see \cite{Antipin1989, Antipin1994, Vuong2020}) dynamical systems. Results regarding the convergence rate and conditions for algorithms derived from second-order dynamical systems typically outperform those from first-order systems. For example, under the strongly monotone condition, the authors in \cite{Vuong2021} demonstrated linear convergence for an IVIP algorithm derived from the discretization of the associated first-order dynamical system, while the algorithm based on a second-order system achieves linear convergence even under the weaker pseudo-strongly monotone condition (see \cite{Vuong2020}). Similarly, for IVIPs, the papers \cite{Vuong2021, Zou2016} utilize a first-order dynamical system and derive projected algorithms through its discretization.\\ 

Due to applications in practice, there is an interest in a more general form of IVIP, that is the inverse quasi variational inequality problem {\rm (IQVIP)}:
Let $\psi : \mathbb{R}^n \rightrightarrows \mathbb{R}^n$ be a set-valued mapping with nonempty, convex, closed point values (($\Gamma_1$-condition, for brief) and $V:\mathbb{R}^n \rightarrow \mathbb{R}^n$ be a single-valued mapping. The IQVIP requires finding $x^*\in \mathbb{R}^n$ such that
\begin{equation}\label{IQVI}
	V(x^*) \in \psi(x^*)  \text{ and } \langle x^*,z-V(x^*) \rangle \geq 0, \quad \forall z\in \psi (x^*).
\end{equation}
Although the theoretical results are developed in \cite{Dey2023, Han2017}, the number of numerical algorithms for solving IQVIPs is still limited. In \cite{Dey2023}, Dey and Reich  used  the \textit{moving set} condition of  $\psi$, which reads as:
\[\psi(x)=\Psi+ h(x) ,\]
where $\Psi$ is a nonempty closed convex subset of $\mathbb{R}^n$ and $h: \mathbb{R}^n \rightarrow \mathbb{R}^n$ is Lipschitz continuous. 
Dey and Reich investigated the global asymptotic stability and exponential stability of a first-order dynamical system. Additionally, they proposed a projection-type algorithm derived from the associated first-order dynamical system and established its linear convergence rate.
Recently, \cite{Thanh2024} demonstrated that the moving set condition is unnecessary for achieving linear convergence, provided appropriate parameters are chosen.
Building on this dynamical systems framework, we extend the study to second-order dynamical systems for solving the IQVIP, aiming to develop more efficient algorithms. 
It is well established in optimization theory that, under comparable assumptions, second-order dynamical systems achieve faster convergence rates. Furthermore, the corresponding algorithms derived from their discretization often result in more efficient numerical methods.
Therefore, we would like to address the following key questions:\\

{\bf Q1:} Does the trajectory generated by the second-order dynamical system converge exponentially?

{\bf Q2:} Can discretization of the second-order dynamical system yield algorithms with linear convergence?\\
The purpose of this paper is to provide affirmative answers to both questions.\\

The structure of the rest of this paper is as follows. In Section 2, we recall some basic definitions, properties and demonstrate a condition for the existence and uniqueness of a solution to the IQVIP. 
In Section 3, we recall conditions for the existence and uniqueness of the trajectory which is a strong global solution of the second order dynamical system, and then show its exponential convergence. 
In Section 4, using the discretization of the second order dynamical system,
 we propose a new algorithm, which is a projection algorithm with inertial effects for solving IQVIP. We show that the iterations converge linearly to the unique solution  under two simple conditions of parameters. 
Finally, we illustrate numerical examples in Section 5 and give some comparison and conclusion on the proposed algorithm.\\

\section{Preliminaries}
We begin by recalling some terminologies used in the whole paper.
 \begin{itemize}
 \item The mapping $V$ is \emph{Lipschitz continuous} with constant $L$ on $\mathbb{R}^n$ if  $\|V(y)-V(z)\|\leq L\|y-z\|$, $\forall y, z\in \mathbb{R}^n$.
  \item The mapping $V$ is \emph{monotone} on $\mathbb{R}^n$ if $\langle V(y)-V(z),y-z\rangle \geq 0$, $\forall y,z \in \mathbb{R}^n$.
 \item The mapping $V$ is \emph{$\eta$-strongly monotone} on $\mathbb{R}^n$ if $\langle V(y)-V(z),y-z\rangle \geq \eta \|y-z\|^2$, $\forall y,z\in \mathbb{R}^n$.\newline
 \end{itemize}
 Let $C$ be nonempty closed convex of $\mathbb{R}^n$ and $x\in \mathbb{R}^n$. Then there exists a unique element $y\in C$ such that
 \[\|x-y\|= \inf_{c\in C} \|x-c\|.\]
 Such point $y$ is denoted as $P_C(x)$ and the mapping $P_C$ is called the \emph{metric projection}. The result below collects some properties of the metric projection. 
\begin{Proposition}[{\cite{GR84}}]\label{ProjectionProperties}
For any $y, z\in \mathbb{R}^n$ and $a\in C$ we have
	\begin{description}
	\item[{\rm (a)}]  $\|P_C(y)-P_C(z)\|\leq\|y-z\|$; 
		\item[{\rm (b)}]  $\langle z-P_C(z),a-P_C(z)\rangle\leq0$;
		\item[{\rm (c)}]  $\|P_C(z)-a\|^2\leq\|z-a\|^2-\|z-P_C(z)\|^2$.
	\end{description}
\end{Proposition}
\begin{Remark}
In the rest of this paper, we assume that the set-valued mapping $\psi:\mathbb{R}^n \rightrightarrows \mathbb{R}^n$ has $(\Gamma_1)$-condition. Then, two following conditions are equivalent:
\begin{itemize}
    \item [\rm{(a)}] $x^*$ is a solution of IQVIP \eqref{IQVI}.
    \item[\rm{(b)}] For any $\mu>0$, $x^*$ is a solution to the projection equation
    \[V(x)=P_{\psi(x)}(V(x)-\mu x).\]
\end{itemize}

Besides, we also assume that $V$ is $L$-Lipschitz continuous and $\eta-$strongly monotone, so we call this $(\Gamma_2)$-condition.

\end{Remark}

The existence and uniqueness of the solution to the IQVIP \eqref{IQVI} will be established through next proposition. The proof is in \cite[Theorem 3.2]{Dey2023}.

\begin{Proposition}[{\cite{Dey2023}}]
Let $\psi: \mathbb{R}^n \rightrightarrows \mathbb{R}^n$ and $V:\mathbb{R}^n \rightarrow \mathbb{R}^n$ be mappings with \rm{($\Gamma_1$)}, \rm{($\Gamma_2$)}-condition, respectively. Assume that there exists $\rho>0$ satisfying
    \begin{equation} \label{Kappa1}
        \|P_{\psi(r)}(y) -P_{\psi(s)}(y)\| \leq \rho \|r-s\|, \quad \forall y,r,s \in \mathbb{R}^n
    \end{equation} 
   and 
   \begin{equation}\label{Kappa2}
       \sqrt{L^2-2\eta \mu+\mu^2}+\rho <\mu,
   \end{equation}
   where $\mu >0$ is a constant. Then the 
 \rm{IQVIP} \eqref{IQVI} has a unique solution.
\end{Proposition}
\begin{Remark}
    The set-valued mapping $\psi$ appears in many applications. For example, we can express the point image as 
    \[\psi(x)=g(x)+\Psi,\]
where $\Psi$ is a closed convex subset in $\mathbb{R}^n$ and $g(x): \mathbb{R}^n\rightarrow \mathbb{R}^n$ is Lipschitz continuous with constant $l$. Then, the set-valued mapping $\psi$ satisfies \eqref{Kappa1} with $\rho = l$.
\end{Remark}
\section{Second-order dynamical system}
In this part, for Lebesgue measurable functions $\sigma, \tau :[0,+\infty) \rightarrow [0,+\infty)$, we propose the following second order dynamical system for solving IQVIP \eqref{IQVI}
\begin{equation}\label{MainDyn}
    \begin{cases}
    \Ddot{x}(t)+\sigma(t) \dot{x}(t)+\tau(t) \left(V(x)-P_{\psi(x)}(V(x)-\mu x)\right)=0,\\
    x(0)=a_0,\dot{x}(0)=b_0.
    \end{cases}
\end{equation}
\subsection{Existence and uniqueness of solution}
In this subsection, we study when the trajectory of the dynamical system \eqref{MainDyn} exists and is unique. We recall the definitions of the absolutely continuous functions and strong global solution of \eqref{MainDyn} (see \cite{Vuong2020}).
\begin{Definition}
    A function $h:[0,q] \rightarrow \mathbb{R}^n$ (where $q>0)$ is an absolutely continuous function if $h$ has one of the following equivalent conditions:
    \item[\rm{(a)}] There exists an integrable function $g:[0,q]\rightarrow \mathbb{R}^n$ fulfilling
    \[h(t)=\int_0^t g(u)du + h(0),\quad \forall t\in [0,q].\]
    \item[\rm{(b)}] $h$ is continuous. Besides, the distributional derivative $\dot{h}$ is Lesbegue integrable on $[0,q]$.
\end{Definition}
\begin{Definition}
    We call a function $x:[0,+\infty)\rightarrow \mathbb{R}^n$ a strong global solution of the dynamical system \eqref{MainDyn} if the following properties hold:
    \item[\rm{(a)}] For any $0<q<+\infty$, $x$ is absolutely continuous on $[0,q]$ ($x$ is also called local absolute continuous function).
    \item[\rm{(b)}] $\Ddot{x}(t)+\sigma(t)\dot{x}(t)+\tau(t)\left(V(x(t))-P_{\psi(x(t))}(V(x(t))-\mu x(t))\right) = 0$ for almost every $t\in [0,+\infty)$.
    \item[\rm{(c)}] $x(0)=a_0$ and $ \dot{x}(0)=b_0$.
\end{Definition}
We are ready to prove the existence and uniqueness of the trajectory to \eqref{MainDyn}. The result below is similar to \cite[Theorem 4]{Bot2016}, but we give a proof, for a completeness of exposition.
\begin{Theorem}\label{UniDyn}
    Let $\sigma, \tau: [0,+\infty) \rightarrow [0,+\infty)$ be Lesbegue measurable functions such that $\sigma, \tau \in L_{loc}^1\left([0,+\infty)\right)$ (that is, $\sigma,\tau \in L_{loc}^1\left([0,s]\right) $ for every $0<s<+\infty$). Let $\psi: \mathbb{R}^n \rightrightarrows \mathbb{R}^n$ and $V:\mathbb{R}^n \rightarrow \mathbb{R}^n$ be mappings with \rm{($\Gamma_1$)}, \rm{($\Gamma_2$)}-condition, respectively. Assume that there exists a number $\rho>0$ such that
    \begin{equation} \label{Lips1}
        \|P_{\psi(r)}(y) -P_{\psi(s)}(y)\| \leq \rho \|r-s\|, \quad \forall y,r,s \in \mathbb{R}^n.
    \end{equation}
    Then for $a_0,b_0 \in \mathbb{R}^n$, there exists a unique strong global solution of the dynamical system \eqref{MainDyn}.
\end{Theorem}
\begin{proof}
Let us define the mapping $B:\mathbb{R}^n \rightarrow \mathbb{R}^n$ by setting
    \[B(x):= V(x)-P_{\psi(x)}(V(x)-\mu x),\]
then we can rewrite equivalently the dynamical system \eqref{MainDyn} as 
\begin{equation}\label{SecDyn2}
    \begin{cases}
    \Ddot{x}(t)+\sigma(t) \dot{x}(t)+\tau(t) B(x(t))=0,\\
    x(0)=a_0,\dot{x}(0)=b_0.
    \end{cases}
\end{equation}
Since $V$ is $L$-Lipschitz continuous, by Theorem \eqref{ProjectionProperties} (a) and \eqref{Lips1}, for all $a,b \in \mathbb{R}^n$ and $\mu>0$, we have
\begin{align}\label{LipB}
   &\|B(a)-B(b)\| \notag \\
   &= \|\left(V(a)-P_{\psi(a)}(V(a)-\mu a)\right)-\left(V(b)-P_{\psi(b)}(V(b)-\mu b )\right)\|\notag\\
    &\leq \|P_{\psi (a)}(V(a)-\mu a )-P_{\psi (a)}(V(b)-\mu b)\|+\|P_{\psi (a)}(V(b)-\mu b )-P_{\psi (b)}(V(b)-\mu b)\| +\|V(a)-V(b)\| \notag\\
    &\leq (L+\rho)\|a-b\|+ \|(V(a)-\mu a)-(V(b)-\mu b)\|\notag\\
    &\leq (2L+\rho+\mu)\|a-b\|;   
\end{align}
which implies that $B$ is Lipschitz continuous with modulus $L_1=2L+\rho+\mu$. We can rewrite equivalently the second-order dynamical system \eqref{MainDyn} as the form of a first-order dynamical system in $\mathbb{R}^n\times \mathbb{R}^n$:
\begin{equation}\label{SecDyn3}
    \begin{cases}
    \dot{G}(t)=K(t,G(t)),\\
    G(0)=(a_0,b_0),
    \end{cases}
\end{equation}
where 
\[G:[0,+\infty) \rightarrow \mathbb{R}^n\times \mathbb{R}^n, \quad G(t)=(x(t),\dot{x}(t))\]
and
\[K:[0,+\infty)\times \mathbb{R}^n\times \mathbb{R}^n \rightarrow \mathbb{R}^n\times \mathbb{R}^n, \quad K(t,u,v) =(v, -\tau(t)B(u)-\sigma (t) v ).\]
In $\mathbb{R}^n\times \mathbb{R}^n$, we endow inner product $\langle(a,b),(c,d)\rangle_{\mathbb{R}^n\times \mathbb{R}^n}= \langle a,c\rangle+\langle b,d\rangle $ and corresponding norm $\|(a,b)\|_{\mathbb{R}^n\times \mathbb{R}^n}=\sqrt{\|a\|^2+\|b\|^2}$.\newline
For any $u_1,u_2,v_1,v_2 \in \mathbb{R}^n$ and $t\geq 0$, since $B$ is $L_1$- Lipschitz, we have
\begin{align*}
   \|K(t,u_2,v_2)-K(t,u_1,v_1)\|_{\mathbb{R}^n\times \mathbb{R}^n} &= \sqrt{\|v_1-v_2\|^2+\|\sigma(t)(v_1-v_2)+\tau(t)(B(u_1)-B(u_2))\|^2 } \\
   & \leq \sqrt{(1+2\sigma^2(t))\|v_1-v_2\|^2+2L_1^2\tau^2(t)\|u_2-u_1\|^2} \\
   & \leq \sqrt{1+2\sigma^2(t)+2L_1^2\tau^2(t)}\|(u_2,v_2)-(u_1,v_1)\|_{\mathbb{R}^n\times \mathbb{R}^n} \\
   &\leq (1+\sqrt{2}\sigma(t)+L_1 \sqrt{2}\tau(t))\|(u_2,v_2)-(u_1,v_1)\|_{\mathbb{R}^n\times \mathbb{R}^n}.
\end{align*}
Besides, since $\sigma, \tau \in L_{loc}^1\left([0,+\infty)\right)$, then for fixed $t>0$, the Lipschitz constant $L_t=1+\sqrt{2}\sigma(t)+L_1 \sqrt{2}\tau(t)$ of $K(t,\cdot,\cdot)$ is locally integrable. \newline
Next, we need prove that
\begin{equation}\label{loc_int}
    \forall p, q \in \mathbb{R}^n, \quad \forall s>0,\quad K(\cdot,p,q)\in L^1\left([0,s],\mathbb{R}^n\times \mathbb{R}^n\right).
\end{equation}
Indeed, with any $p,q\in \mathbb{R}^n$ and $s>0$, we have
\begin{align*}
    \int_0^s\|K(t,p,q)\|_{\mathbb{R}^n\times \mathbb{R}^n}dt &= \int_0^s\sqrt{\|q\|^2+\|\sigma(t)q+\tau(t)B(p)\|^2}dt \\
    &\leq \int_0^s\sqrt{(1+2\sigma^2(t))\|q\|^2+2\tau^2(t)\|B(p)\|^2}dt \\
    & \leq \int_0^s \left((1+\sqrt{2}\sigma(t))\|q\|+\sqrt{2}\tau(t)\|B(p)\|\right)dt.
\end{align*}
By using the assumptions made on $\sigma$ and $\tau$, we get \eqref{loc_int}. \newline
In view of the Cauchy-Lipschitz-Picard theorem (see, \cite{Haraux1991}, Proposition 6.2.1), we receive the existence and uniqueness of the strong global solution to \eqref{SecDyn3}. Because \eqref{MainDyn}, \eqref{SecDyn2} and \eqref{SecDyn3} are equivalent, we have the desired result.
\end{proof}
\subsection{Exponential convergence}

Before showing the exponential convergence of trajectory $x(t)$ generated by dynamical system \eqref{MainDyn}, we need the following result, which will play an important role  in this convergence analysis. 
\begin{Proposition} 
Let $\psi: \mathbb{R}^n \rightrightarrows \mathbb{R}^n$ and $V:\mathbb{R}^n \rightarrow \mathbb{R}^n$ be mappings with \rm{($\Gamma_1$)}, \rm{($\Gamma_2$)}-condition, respectively. Assume that conditions \eqref{Kappa1}, \eqref{Kappa2} hold and
\[\theta:=\eta-\rho-\frac{1}{2}-\frac{1}{2}L^2-\frac{1}{2}\mu^2+\mu\eta>0.\]
Let $x^*$ be a unique solution of the IQVIP \eqref{IQVI}. For all $w\in \mathbb{R}^n$ we have
\begin{equation*}
    \theta_1\|V(w)-P_{\psi(w)}(V(w)-\mu w )\|^2 \leq \langle V(w)-P_{\psi(w)}(V(w)-\mu w ), w-x^* \rangle,
\end{equation*}
where $\theta_1= \frac{\theta}{(2L+\rho+\mu)^2}$ and
\begin{equation*}
    \theta \|w-x^*\|\leq \|V(w)-P_{\psi(w)}(V(w)-\mu w )\|.
\end{equation*}
\end{Proposition}

\begin{proof}
 One the one hand, as shown in \eqref{LipB}, we have
\begin{align}\label{A}
\|V(w)-P_{\psi(w)}(V(w)-\mu w)\|&=\|\left(V(w)-P_{\psi(w)}(V(w)-\mu w)\right)-\left(V(x^*)-P_{\psi(x^*)}(V(x^*)-\mu x^* )\right)\|\notag\\
 &\leq (2L+\rho+\mu)\|w-x^*\|.
\end{align}
On the other hand, we have
\begin{align}\label{B}
&\langle V(w)-P_{\psi(w)}(V(w)-\mu w),w-x^*\rangle\notag\\
& = \langle V(w)-V(x^*),w-x^*\rangle - \langle P_{\psi (w)}(V(w)-\mu w )-P_{\psi (w)}(V(x^*)-\mu x^*),w-x^*\rangle \notag \\
\quad &- \langle P_{\psi (w)}(V(x^*)-\mu x^* )-P_{\psi (x^*)}(V(x^*)-\mu x^*),w-x^*\rangle \notag \\
&\geq \left(\eta-\rho -\frac{1}{2}\right) \|w-x^*\|^2-\frac{1}{2}\|P_{\psi (w)}(V(w)-\mu w )-P_{\psi (w)}(V(x^*)-\mu x^*)\|^2 \notag \\
& \geq \left(\eta-\rho -\frac{1}{2}\right)\|w-x^*\|^2 - \frac{1}{2}\|(V(w)- \mu w) - (V(x^*)-\mu x^*)\|^2 \notag \\
&\geq \left(\eta-\rho-\frac{1}{2}-\frac{1}{2}L^2-\frac{1}{2}\mu^2+\mu\eta\right)\|w-x^*\|^2.
\end{align}
Combining \eqref{A} and \eqref{B}, we get
\begin{align}\label{C}
   \theta_1 \|V(w)-P_{\psi (w)}(V(w)-\mu w)\|^2 &= \frac{\eta-\rho-\frac{1}{2}-\frac{1}{2}L^2-\frac{1}{2}\mu^2+\mu\eta}{(2L+\rho+\mu)^2}\|V(w)-P_{\psi (w)}(V(w)-\mu w)\|^2 \notag \\
   &\leq \langle V(w)-P_{\psi(w)}(V(w)-\mu w),w-x^*\rangle. 
\end{align}
We also have from \eqref{B} that
\begin{align}\label{CC}
\left(\eta-\rho-\frac{1}{2}-\frac{1}{2}L^2-\frac{1}{2}\mu^2+\mu\eta\right)\|w-x^*\|^2 &\leq \langle V(w)-P_{\psi(w)}(V(w)-\mu w),w-x^*\rangle \\
    &\leq  \|V(w)-P_{\psi(w)}(V(w)-\mu w)\|\cdot\|w-x^*\|. \notag
\end{align}
Thus
\begin{equation}\label{D}
   \theta\|w-x^*\|\leq   \|V(w)-P_{\psi(w)}(V(w)-\mu w)\|.
\end{equation}   
\end{proof}

\bigskip

The main result of this section is as follows.
\begin{Theorem}
Let $\psi: \mathbb{R}^n \rightrightarrows \mathbb{R}^n$ and $V:\mathbb{R}^n \rightarrow \mathbb{R}^n$ be mappings with \rm{($\Gamma_1$)}, \rm{($\Gamma_2$)}-condition, respectively. Assume that
\[\theta:=\eta-\rho-\frac{1}{2}-\frac{1}{2}L^2-\frac{1}{2}\mu^2+\mu\eta>0,\]
and 
\begin{equation}\label{UniIqvi}
    \sqrt{L^2-2\eta\mu+\mu^2}+\rho <\mu,
\end{equation}
where $\rho$ satisfies
\begin{equation}\label{Kappa3}
    \|P_{\psi(r)}(y) -P_{\psi(s)}(y)\| \leq \rho \|r-s\|, \quad \forall y,r,s \in \mathbb{R}^n.
\end{equation}

Let $\sigma, \tau :[0,+\infty)\rightarrow [0,+\infty)$ be locally absolutely continuous functions satisfying for every $t\in [0,+\infty)$ that
\item {(i)} $1<\sigma\leq \sigma(t)\leq \theta^2\theta_1\tau(t)+1$;
\item{(ii)} $\dfrac{d}{dt}\left(\dfrac{\sigma(t)}{\tau(t)}\right)\leq 0$ and $\dot{\sigma}(t)\leq 0$;
\item{(iii)} $\sigma^2(t)-\sigma(t)- \dfrac{2\tau(t)}{\theta_1}\geq 0$. \newline
Then the trajectory $x(t)$ generated by the dynamical system \eqref{MainDyn} converges exponentially to $x^*$ as $t\rightarrow +\infty$ where $x^*$ is the unique solution of the IQVIP \eqref{IQVI}, i.e., there exist positive constants $\nu,\zeta$ such that
\[\|x(t)-x^*\|\leq \nu\|x(0)-x^*\|e^{-\zeta t},\, \forall t\geq 0.\]
\end{Theorem}
\begin {proof}
Using the assumption \eqref{Kappa3} and Theorem \ref{UniDyn}, we can show that the dynamical system \eqref{MainDyn} has a unique strong global solution. Besides, using \eqref{Kappa3} and \eqref{UniIqvi}, we obtain that the IQVIP \eqref{IQVI} has a unique solution $x^*$.

We consider the function $k(t)=\frac{1}{2}\|x(t)-x^*\|^2$ for every $t \in [0, + \infty)$. Then 
\[\dot{k} (t)=\langle x(t)-x^*,\dot{x}\rangle; \quad \Ddot{k}(t)= \langle x(t)-x^*,\Ddot{x}\rangle +\|\dot{x}(t)\|^2.\]
Taking the dynamical system \eqref{MainDyn} into account, for any $t\geq 0$ we get
\[\Ddot{k}(t)+\sigma(t)\dot{k}(t)+\tau(t)\langle V(x(t))-P_{\psi (x(t))}(V(x(t))-\mu x),x(t)-x^* \rangle - \|\dot{x} (t)\|^2 =0.\]
Combining with \eqref{C}, it yields
\[\Ddot{k}(t)+\sigma(t)\dot{k}(t)+\theta_1\tau(t)\| V(x(t))-P_{\psi (x(t))}(V(x(t))-\mu x(t))\|^2 - \|\dot{x} (t)\|^2 \leq 0.\]
From the dynamical system \eqref{MainDyn}, we infer
\[\Ddot{k}(t)+\sigma(t)\dot{k}(t)+\frac{1}{2}\theta_1\tau(t)\| V(x(t))-P_{\psi (x(t))}(V(x(t))-\mu x(t))\|^2+\frac{\theta_1}{2\tau(t)}\|\Ddot{x}(t)+\sigma(t)\dot{x}(t)\|^2 - \|\dot{x} (t)\|^2 \leq 0.\]
Using \eqref{D} we can deduce
\begin{equation}\label{6524}
\Ddot{k}(t)+\sigma(t)\dot{k}(t)+\theta^2\theta_1\tau(t) k(t)+\frac{\theta_1}{2\tau(t)}\|\Ddot{x}(t)\|^2+\frac{\theta_1\sigma(t)}{\tau(t)}\langle \Ddot{x}(t),\dot{x} (t)\rangle+\left(\frac{\theta_1\sigma^2(t)}{2\tau(t)}-1\right)\|\dot{x}(t)\|^2  \leq 0.
\end{equation}
Noticing that $\frac{\theta_1}{2\tau(t)}\|\Ddot{x}(t)\|^2\geq 0$ and $\langle \Ddot{x}(t),\dot{x} (t)\rangle = \dfrac{1}{2}\dfrac{d}{dt}\|\dot{x}(t)\|^2$, setting 
\[\alpha(t)=\theta^2\theta_1\tau(t), \quad \beta(t)=\frac{\theta_1\sigma(t)}{2\tau(t)},\quad \chi(t)=\frac{\theta_1\sigma^2(t)}{2\tau(t)}-1, \quad  \delta(t)=\|\dot{x}(t)\|^2, \quad \text{for every}\, t\in [0,+\infty), \]
we can rewrite \eqref{6524} as
\begin{equation}\label{E}
  \Ddot{k}(t)+\sigma(t)\dot{k}(t) +\alpha(t)k(t)+ \chi(t)\delta(t)+ \beta(t)\dot{\delta}(t)\leq 0.
\end{equation}
Besides, we have
\begin{align*}
    e^t\Ddot{k}(t)&
    =\dfrac{d}{dt}\left(e^t\dot{k}(t)\right)-\dfrac{d}{dt}\left(e^t k(t)\right)+e^t k(t) ,\notag \\
    \sigma (t)e^t\dot{k}(t)&= \sigma(t) \dfrac{d}{dt} \left(e^t k(t)\right)- \sigma(t)e^t k(t), \\
    \beta(t)e^t \dot{\delta}(t) &=\beta(t)\dfrac{d}{dt}\left(e^t\delta(t)\right)-\beta(t)e^t \delta(t).
\end{align*}
Multiplying both sides of \eqref{E} with $e^t$, and using above identities, we obtain
\begin{align}\label{F}
    \dfrac{d}{dt}\left(e^t\dot{k}(t)\right)+(\sigma(t)-1)\dfrac{d}{dt}\left(e^t k(t)\right)+ \left( \alpha(t)+1-\sigma(t)\right)e^t k(t) \notag\\ +(\chi(t)-\beta(t))e^t \delta(t)+\beta(t)\dfrac{d}{dt}\left(e^t \delta(t)\right)\leq 0.
\end{align}
It follows from the conditions (i) and (iii)  that  
\[\quad \chi(t)-\beta(t) \geq 0, \quad \alpha(t)+1-\sigma(t) \geq 0,  \quad \forall t \in [0,+\infty). \]
Therefore, from \eqref{F} we get
\begin{equation}\label{G}
   \dfrac{d}{dt}\left(e^t\dot{k}(t)\right) + (\sigma(t)-1)\dfrac{d}{dt}\left(e^t k(t)\right)+\beta(t)\dfrac{d}{dt}\left(e^t \delta(t)\right) \leq 0 .
\end{equation}
Using identities 
\begin{align*}
    (\sigma(t)-1)\dfrac{d}{dt}\left(e^t k(t)\right)&=\dfrac{d}{dt}\left[(\sigma(t)-1)e^t k(t)\right]-\dot{\sigma}(t)e^t k(t),\\
    \beta(t)\dfrac{d}{dt}\left( e^t \delta(t)\right)&=\dfrac{d}{dt}\left[\beta(t)e^t \delta(t)\right]-\dot{\beta}(t)e^t \delta(t),
\end{align*}
from \eqref{G} we obtain
\begin{equation}\label{H}
    \dfrac{d}{dt}\left(e^t \dot{k}(t)\right) +\dfrac{d}{dt}\left[(\sigma(t)-1)e^t k(t)\right]-\dot{\beta}(t)e^t \delta(t)-\dot{\sigma}(t)e^t k(t)+\dfrac{d}{dt}\left(\beta(t)e^t \delta(t)\right)\leq 0.
\end{equation}
Furthermore, we also have $\dot{\sigma}(t)\leq 0$ from assumption (ii) and $\dot{\beta}(t)\leq 0$. Thus, it follows from \eqref{H} that
\[\dfrac{d}{dt}\left[e^t \dot{k}(t)+\beta(t)e^t \delta(t)+(\sigma(t)-1)e^t k(t)\right]\leq 0, \]
which implies that the function defined by
\[t\mapsto e^t \dot {k} (t)+\beta(t)e^t \delta(t) +(\sigma(t)-1)e^t k(t)\]
is non-increasing. Therefore, there exists $T>0$ such that for any $t\in [0,+\infty)$
\[e^t \dot {k} (t)+\beta(t)e^t \delta(t) +(\sigma(t)-1)e^t k(t)\leq T.\]
Because of $\beta(t), \delta(t)\geq 0,\, \forall t\in [0,+\infty)$, we obtain
\[\dot{k} (t)+(\sigma(t)-1)k(t)\leq Te^{-t};\]
hence
\[\dot{k} (t)+(\sigma-1)k(t)\leq Te^{-t}.\]
Multiplying both sides of last inequality with $e^{(\sigma-1)t}>0$, we deduce
\[\dfrac{d}{dt}\left[e^{(\sigma-1)t}k(t)\right]\leq T e^{(\sigma-2)t}\]
for every $t\in [0,+\infty)$. Using integration, we have three following cases:\\
(a) if $1<\sigma<2$ then
\[e^{(\sigma-1)t}k(t)\leq \frac{T}{\sigma-2}\left[e^{(\sigma-2)t}-1\right]+k(0)\leq \frac{T}{2-\sigma}+k(0),\]
which implies
\[k(t)\leq e^{-(\sigma -1)t}\left[\frac{T}{2-\sigma}+k(0)\right];\]
(b) if $\sigma>2$ then 
\[e^{(\sigma-1)t}k(t)\leq \frac{T}{\sigma-2}\left[e^{(\sigma-2)t}-1\right]+k(0)\leq \frac{T}{\sigma-2}e^{(\sigma-2)t}+k(0),\]
which implies
\[k(t)\leq \frac{T}{\sigma-2}e^{-t}+k(0)e^{-(\sigma-1)t}\leq e^{-t}\left(\frac{T}{\sigma-2}+k(0)\right);\]
(c) if $\sigma =2$ then
\[0\leq k(t) \leq e^{-t}(k(0)+Tt).\]
From what have been shown, we conclude that $x(t)$ converges exponentially to $x^*$.
\end{proof}
\begin{Remark}
We verify that there exist functions $\sigma(t)$ and $\tau(t)$ satisfying conditions (i)-(iii). With $\sigma, \tau>1$, let $\sigma(t)=\sigma +\frac{1}{t+1}$ and $\tau(t)=\tau - \frac{1}{t+1}$ where $t\in [0,+\infty)$. By simple calculating, we can verify the condition (ii). Let us consider the condition (i) as
\[\sigma(t) \leq \theta^2\theta_1 \tau(t)+1,\]
or equivalently
\[\sigma+\frac{1}{t+1} \leq \theta^2\theta_1\left(\tau-\frac{1}{t+1}\right)+1,\quad \forall t\in [0,+\infty).\]
We deduce
\[\sigma \leq \theta^2\theta_1 \tau-(1+\theta^2\theta_1)\frac{1}{t+1}+1, \quad \forall t\in [0,+\infty).\]
Thus, we need
\begin{equation}\label{betacon1}
    \sigma \leq \theta^2\theta_1\tau -(1+\theta^2\theta_1)+1=\theta^2\theta_1(\tau-1). 
\end{equation}
If the condition (iii) is fulfilled, then we get
\[\left(\sigma+\frac{1}{t+1}\right)^2-\left(\sigma +\frac{1}{t+1}\right)-\frac{2}{\theta_1}\left(\tau -\frac{1}{t+1}\right) \geq 0,\]
which is equivalent to
\[\sigma^2-\sigma-\frac{2}{\theta_1}\tau +\frac{1}{(t+1)^2} +\frac{2}{\theta_1(t+1)}+\frac{2\sigma-1}{t+1}\geq 0,\, \forall t\in [0,+\infty). \]
Noting that $\sigma \geq 1$, hence we just need
\[\sigma^2-\sigma-\frac{2}{\theta_1}\tau \geq 0,\]
or
\begin{equation}\label{betacon2}
    \sigma \geq \frac{1}{2}+ \frac{1}{2}\sqrt{1+\frac{8\tau}{\theta_1}}.
\end{equation}
It follows from \eqref{betacon1} and \eqref{betacon2} that, if we choose $\tau$ large enough and $\sigma$ such that
\begin{equation}\label{betacon3}
\frac{1}{2}+ \frac{1}{2}\sqrt{1+\frac{8\tau}{\theta_1}} \leq \sigma \leq \theta^2\theta_1 (\tau-1),
\end{equation}
then the conditions (i)-(iii) hold.\\
On the other hand, if we choose $\sigma(t) =\sigma, \tau(t)=\tau$ for every $t\in [0,+\infty)$, where $\sigma, \tau$ satisfy \eqref{betacon3}, then the conditions (i)-(iii) still hold.
\end{Remark}
\section{Discretization of the dynamical system}
We consider the explicit discretization of dynamical system with respect to $t$ with the step size $h_n>0$, relaxation variable $\tau_n>0$, damping variable $\sigma_k>0$ and initial points $x_0$ and $x_1$ in $\mathbb {R}^n$ as follows:
\begin{equation*} 
    \frac{x_{n+1}-2x_n+x_{n-1}}{h_n^2}+\sigma_n \frac{x_n-x_{n-1}}{h_n}+\tau_n \left(V(x_n)-P_{\psi (x_n)}(V(x_n)-\mu x_n)\right)=0
\end{equation*}
or equivalently
\begin{equation*}
    x_{n+1}=x_n+(1-\sigma_n h_n)(x_n-x_{n-1})+\tau_n h_n^2 \left(P_{\psi (x_n)}(V(x_n)-\mu x_n)-V(x_n)\right).
\end{equation*}
If $h_n=1$, $\sigma_n,\tau_n$ are positive constants, we can write the above scheme as 
\begin{equation}\label{202403251017}
    \begin{cases}
        y_n:=x_n+(1-\sigma)(x_n-x_{n-1}),\\
        x_{n+1}=y_n+\tau\left(P_{\psi (x_n)}(V(x_n)-\mu x_n)-V(x_n)\right),
    \end{cases}
\end{equation}
which is a projection algorithm with an inertial effect term 
$(1-\sigma) (x_{n} - x_{n-1})$.\\

If $\sigma =1 $ then \eqref{202403251017} 
reduces to the projection algorithm 
\begin{equation}\label{1stAlg}
    x_{n+1}=x_n+\tau\left[P_{\psi(x_n)}(V(x_n)-\mu x_n)-V(x_n) \right].
\end{equation}
which was obtained from the discretization of a first order dynamical system studied in \cite{Thanh2024}.

Recall the operation of difference and its properties used in the convergence analysis.
\begin{gather*}
x^\dla(n):=x_{n+1}-x_n,\quad x^\nabla(n):=x_n-x_{n-1},\\
x^{\dla\nabla}:=(x^\dla)^\nabla,\quad x^{\nabla\dla}:=(x^\nabla)^\dla.
\end{gather*}

\begin{Proposition}\label{rem20220120}
It holds that
\begin{gather*}
\inner{h}{g}^\dla(n)=\inner{h^\dla(n)}{g_n}+\inner{h_n}{g^\dla(n)}+\inner{h^\dla(n)}{g^\dla(n)},\\
\inner{h}{g}^\nabla(n)=\inner{h^\nabla(n)}{g_n}+\inner{h_n}{g^\nabla(n)}-\inner{h^\nabla(n)}{g^\nabla(n)},\\
x^{\dla\nabla}(n)=x^{\nabla\dla}(n)=x_{n+1}-2x_{n}+x_{n-1}=x^\dla(n)-x^\nabla(n).
\end{gather*}
\end{Proposition}

Using the difference operations, we can rewrite \eqref{202403251017} as follows
\begin{gather}\label{202403251018}
x^{\dla\nabla}(n)+\sigma x^\nabla(n)=\tau\left(P_{\psi (x_n)}(V(x_n)-\mu x_n)-V(x_n)\right).    
\end{gather}

Denote
\begin{gather*}
    a_n:=\|x^\dla(n)\|^2,\quad c_n:=\|x^\nabla(n)\|^2,\quad v_n:=\|x_n-x^*\|^2.
\end{gather*}


\begin{Theorem}\label{thm20240326}
Let $\psi: \mathbb{R}^n \rightrightarrows \mathbb{R}^n$ and $V:\mathbb{R}^n \rightarrow \mathbb{R}^n$ be mappings with \rm{($\Gamma_1$)}, \rm{($\Gamma_2$)}-condition, respectively. Assume that
    \item[(A1)]
\[\theta:=\eta-\rho-\frac{1}{2}-\frac{1}{2}L^2-\frac{1}{2}\mu^2+\mu\eta>0, \quad \sqrt{L^2-2\eta \mu+\mu^2}+\rho <\mu,\]
where $\rho$ satisfies
\[ \|P_{\psi(r)}(y) -P_{\psi(s)}(y)\| \leq \rho \|r-s\|, \quad \forall y,r,s \in \mathbb{R}^n\]
Denote
\begin{gather*}
    \theta_1:=\frac{\theta}{(2L+\rho+\mu)^2}.
\end{gather*}
Furthermore, coefficients $\sigma,\tau$ satisfy the following conditions
\item[(B1)]
\begin{gather*}
    0<\sigma<1.
\end{gather*}
\item[(B2)]
\begin{gather*}
    0<\tau<\theta_1\cdot\min\left\{\frac{1-\sigma}{4},\frac{\sigma^2}{4-\sigma}\right\}.
\end{gather*}
Then the sequence $\{x_n\}$ of \eqref{202403251018} converges linearly to the unique solution of the IQVIP \eqref{IQVI}.    
\end{Theorem}
\begin{proof}
It follows from (B1)-(B2) that 
\begin{gather}\label{(A2)}
    \frac{\theta_1}{\tau}(1-\sigma)\geq 4
\end{gather}
and furthermore
\begin{gather*}
    \sigma\left(\frac{\theta_1\sigma}{\tau}+1\right)- 4>0.
\end{gather*}
Hence, we have
\begin{gather*}
\lim\limits_{\varepsilon\to 1^+}\left[\varepsilon\left(\sigma\left(\frac{\theta_1\sigma}{\tau}+1\right)- 4\right)-(\varepsilon-1)\left(\frac{\theta_1\sigma}{\tau}+1\right)\right]=\sigma\left(\frac{\theta_1\sigma}{\tau}+1\right)- 4>0,\\
\lim\limits_{\varepsilon\to 1^+}\left[\varepsilon\theta\tau-\sigma\varepsilon(\varepsilon-1)+(\varepsilon-1)^2\right]=\theta\tau>0,\\
\lim\limits_{\varepsilon\to 1^+}\left[1-\varepsilon^2(1-\sigma)\right]=1-(1-\sigma)=\sigma>0.
\end{gather*}
Thus, there exists $\varepsilon>1$ subject to the followings
\begin{gather}
\label{A3}\varepsilon\left(\sigma\left(\frac{\theta_1\sigma}{\tau}+1\right)- 4\right)-(\varepsilon-1)\left(\frac{\theta_1\sigma}{\tau}+1\right)>0,\\
\label{A4}\varepsilon\theta\tau-\sigma\varepsilon(\varepsilon-1)+(\varepsilon-1)^2>0,\\
\label{A5}1-\varepsilon^2(1-\sigma)>0.
\end{gather}
Denote
\begin{gather*}
    C_0:=\sigma\left(\frac{\theta_1}{\tau}\sigma+1\right)-4,\quad C_1:=\frac{\theta_1}{\tau}\sigma+1.
\end{gather*}
We have
\begin{gather*}
    v^{\dla\nabla}(n)+\sigma v^\nabla(n)=2\tau\inner{P_{\psi (x_n)}(V(x_n)-\mu x_n)-V(x_n)}{x_n-x^*}+2a_n-\sigma c_n-c^\dla(n),
\end{gather*}
which implies, by \eqref{CC} and \eqref{C}, that
\begin{gather*}
v^{\dla\nabla}(n)+\sigma v^\nabla(n)
\leq-\theta_1\tau\|P_{\psi (x_n)}(V(x_n)-\mu x_n)-V(x_n)\|^2-\theta\tau v_n+2a_n-\sigma c_n-c^\dla(n).
\end{gather*}
Then
\begin{gather*}
v^{\dla\nabla}(n)+\sigma v^\nabla(n)+\theta\tau v_n+\sigma c_n+c^\dla(n)\\
\leq-\theta_1\tau\|P_{\psi (x_n)}(V(x_n)-\mu x_n)-V(x_n)\|^2+2a_n\\
=-\frac{\theta_1}{\tau}\|x^{\dla\nabla}(n)+\sigma x^\nabla(n)\|^2+2a_n\\
=-\frac{\theta_1}{\tau}\left(\|x^{\dla\nabla}(n)\|^2+\sigma^2 c_n+2\sigma\inner{x^{\dla\nabla}(n)}{x^\nabla(n)}\right)+2a_n\\
=-\frac{\theta_1}{\tau}\left(\|x^{\dla\nabla}(n)\|^2(1-\sigma)+\sigma^2c_n+\sigma c^\dla(n)\right)+2a_n.
\end{gather*}
We infer
\begin{gather*}
v^{\dla\nabla}(n)+\sigma v^\nabla(n)+\theta\tau v_n+\sigma\left(\frac{\theta_1}{\tau}\sigma+1\right)c_n+\left(\frac{\theta_1}{\tau}\sigma+1\right)c^\dla(n)\\
\leq-\frac{\theta_1}{\tau}(1-\sigma)\|x^{\dla\nabla}(n)\|^2+2a_n,
\end{gather*}
which implies, as
\begin{gather*}
a_n=\|x^\dla(n)-x^\nabla(n)+x^\nabla(n)\|^2=\|x^{\dla\nabla}(n)+x^\nabla(n)\|^2\\
\leq 2(\|x^{\dla\nabla}(n)\|^2+\|x^\nabla(n)\|^2),
\end{gather*}
that
\begin{gather*}
v^{\dla\nabla}(n)+\sigma v^\nabla(n)+\theta\tau v_n+C_0c_n+C_1c^\dla(n)\\
\leq\left(4-\frac{\theta_1}{\tau}(1-\sigma)\right)\|x^{\dla\nabla}(n)\|^2\\
\leq 0\quad\text{(by \eqref{(A2)})}.
\end{gather*}
Multiplying both sides by $\varepsilon^{n+1}$ and then using Proposition \ref{rem20220120}, we get
\begin{gather*}
0\geq\varepsilon^{n+1}\left(v^{\dla\nabla}(n)+\sigma v^\nabla(n)+\theta\tau v_n+C_0c_n+C_1c^\dla(n)\right)\\
=(\varepsilon^nv^\nabla)^\dla(n)+(\varepsilon\sigma-\varepsilon+1)(\varepsilon^{n+1}v)^\nabla(n)+\varepsilon^nv_n[\varepsilon\theta\tau-\sigma\varepsilon(\varepsilon-1)+(\varepsilon-1)^2]\\
+C_1(\varepsilon^nc)^\dla(n)+\varepsilon^nc(n)[\varepsilon C_0-(\varepsilon-1)C_1].
\end{gather*}
Using \eqref{A3}-\eqref{A4}, the inequality above gives
\begin{gather*}
0\geq(\varepsilon^nv^\nabla)^\dla(n)+(\varepsilon\sigma-\varepsilon+1)(\varepsilon^{n+1}v)^\nabla(n)+C_1(\varepsilon^nc)^\dla(n).
\end{gather*}
Summing the line above from $n=1$ to $n=m$, we obtain
\begin{gather*}
M_1\geq\varepsilon^{m+1}v^\nabla(m+1)+(\varepsilon\sigma-\varepsilon+1)\varepsilon^{m+1}v_m+C_1\varepsilon^{m+1}c_{m+1},
\end{gather*}
where $M_1$ is some positive constant. Since $C_1>0$, we infer
\begin{gather*}
M_1\geq\varepsilon^{m+1}v^\nabla(m+1)+(\varepsilon\sigma-\varepsilon+1)\varepsilon^{m+1}v_m\\
=\varepsilon^{m+1}v^\dla(m)+(\varepsilon\sigma-\varepsilon+1)\varepsilon^{m+1}v_m\\
=(\varepsilon^mv_m)^\dla+\varepsilon^mv_m[\varepsilon^2(\sigma-1)+1]\\
\geq(\varepsilon^mv_m)^\dla,
\end{gather*}
where the last inequality uses \eqref{A5}. Summing this inequality from $m=1$ to $m=p$, we see
\begin{gather*}
M_1p+\varepsilon v_1\geq\varepsilon^{p+1}v_{p+1},    
\end{gather*}
which means that the sequence $\{x_n\}$ generated by \eqref{202403251018} converges linearly to the unique solution of the IQVIP \eqref{IQVI}.    
\end{proof}

\section{Numerical Experiment}
In this section, we first consider some academic examples in $\mathbb{R}^2$ and then discuss the applications to traffic assignment problems. We also compare the performance of algorithm \eqref{202403251017} with the algorithm derived from the first order dynamical system (see \cite{Thanh2024}).\\
\begin{Example}
 In $\mathbb{R}^2$, let $V: \mathbb{R}^2\rightarrow \mathbb{R}^2$ defined by $V(x)=Q x$ with $Q= \begin{pmatrix}
     3.4& -0.64\\
     2.375 & 0.8
 \end{pmatrix}$. Let $\psi: \mathbb{R}^2 \rightrightarrows \mathbb{R}^2$ be mapping with $(\Gamma_1)$-condition defined by: for $(x_1,x_2)\in \mathbb{R}^2$, $\psi(x_1,x_2)$ is the rectangular constructed by four lines: $x=x_1, x=0, y=y_2, y=0$. We consider the IQVIP \eqref{IQVI} and use the algorithm \eqref{202403251017} for solving it..
\end{Example}
Note that the matrix $Q$ is positive definite. 
Thus, the mapping $V$ is Lipschitz continuous with $L=2.2$, which is
the maximum eigenvalue of $Q$, and strongly monotone with $\eta =2$, which is the 
minimum eigenvalue of $Q$. With $\mu =2$, we need to verify the existence and uniqueness of the IQVIP \eqref{IQVI}. For $y,r,s \in \mathbb{R}^2$, we have
\begin{equation*}
    \|P_{\psi(y)}(r)-P_{\psi (y)}(s)\| \leq \|r-s\|,
\end{equation*}
which means the set-valued mapping $\psi$ satisfies the condition \eqref{Kappa1}. Next, we check the existence of the solution to the IQVIP \eqref{IQVI} by calculating
\[\theta:= \eta-\rho-\frac{1}{2}-\frac{1}{2}L^2-\frac{1}{2}\mu^2+\mu\eta = 0.08 >0.\]
It can be seen that the pair $(0,0)$ is a solution of IQVIP \eqref{IQVI}. Besides, we also have 
\[ \mu-\sqrt{L^2-2\eta\mu+\mu^2} -\rho \approx 0.083 >0,\]
 hence, the IQVIP \eqref{IQVI} has a unique solution, which is $(0,0)$. We calculate
\[\theta_1=\frac{\theta}{(2L+\rho+\mu)^2}\approx 0.00146.\]
Now we are ready to apply the algorithm \eqref{202403251017}.  The performance of the algorithm \eqref{202403251017} is illustrated in Figure \ref{fig1}. Noting that, from condition $(B2)$, we expect to choose the value for $\tau$ as big as possible. Therefore, we choose $\sigma =0.59$ and the corresponding $\tau=0.000146<0.1\theta_1$. With the beginning points $(7,5)$ and $ (7,-5)$ and let the $x_{-1}= x_0$, the sequences generated by the algorithm \eqref{202403251017} and its norm (see Figure \ref{fig2}) are marked by red and black line. If we choose $\sigma =0.9$, which makes the inertial effect parameter $1-\sigma$ smaller and start with point $(-7,5)$, the result is marked by green line. By contraction, we choose $\sigma =0.1$ to make the inertial parameter greater and show the sequence and its norm as blue line.\\
In Figures \ref{fig1} and \ref{fig2}, if we reduce the inertial parameter, the algorithm \eqref{202403251017} runs slower, even can not reach the 0.1 error within 20000 step. But if we take the inertial parameter $(1-\sigma)$ too big as in blue line, then the algorithm \eqref{202403251017} can not keep linear convergence because condition $(B2)$ fails to hold.\\
\begin{figure}
\centering
\includegraphics[scale=0.73]{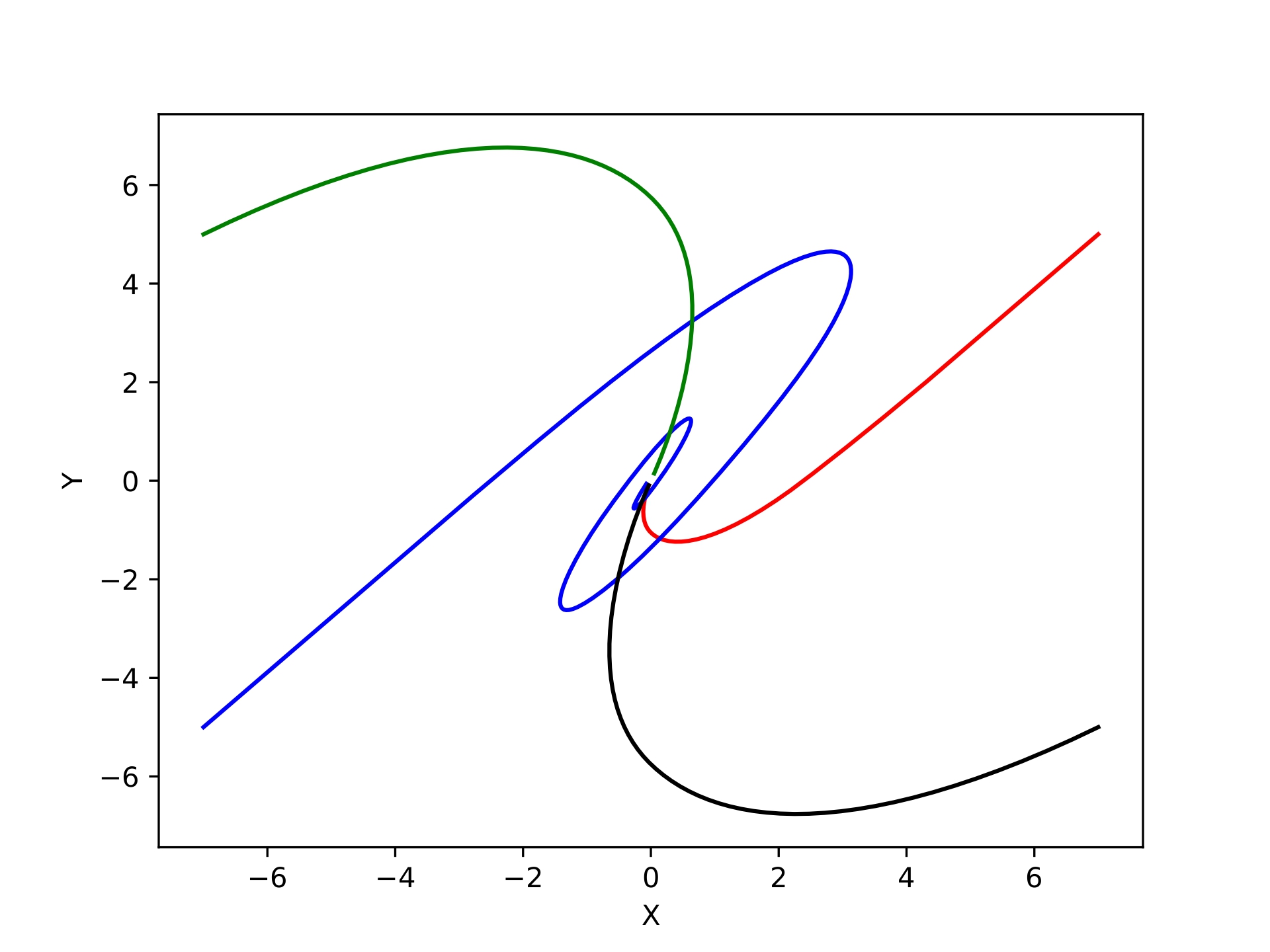}\\ 
\caption{Convergence of sequences generated by algorithm \eqref{202403251017} with different starting points and different inertial parameter $1-\sigma$.} 
\label{fig1}
\end{figure}

\begin{figure}[ht]
\centering
\includegraphics[scale=0.73]{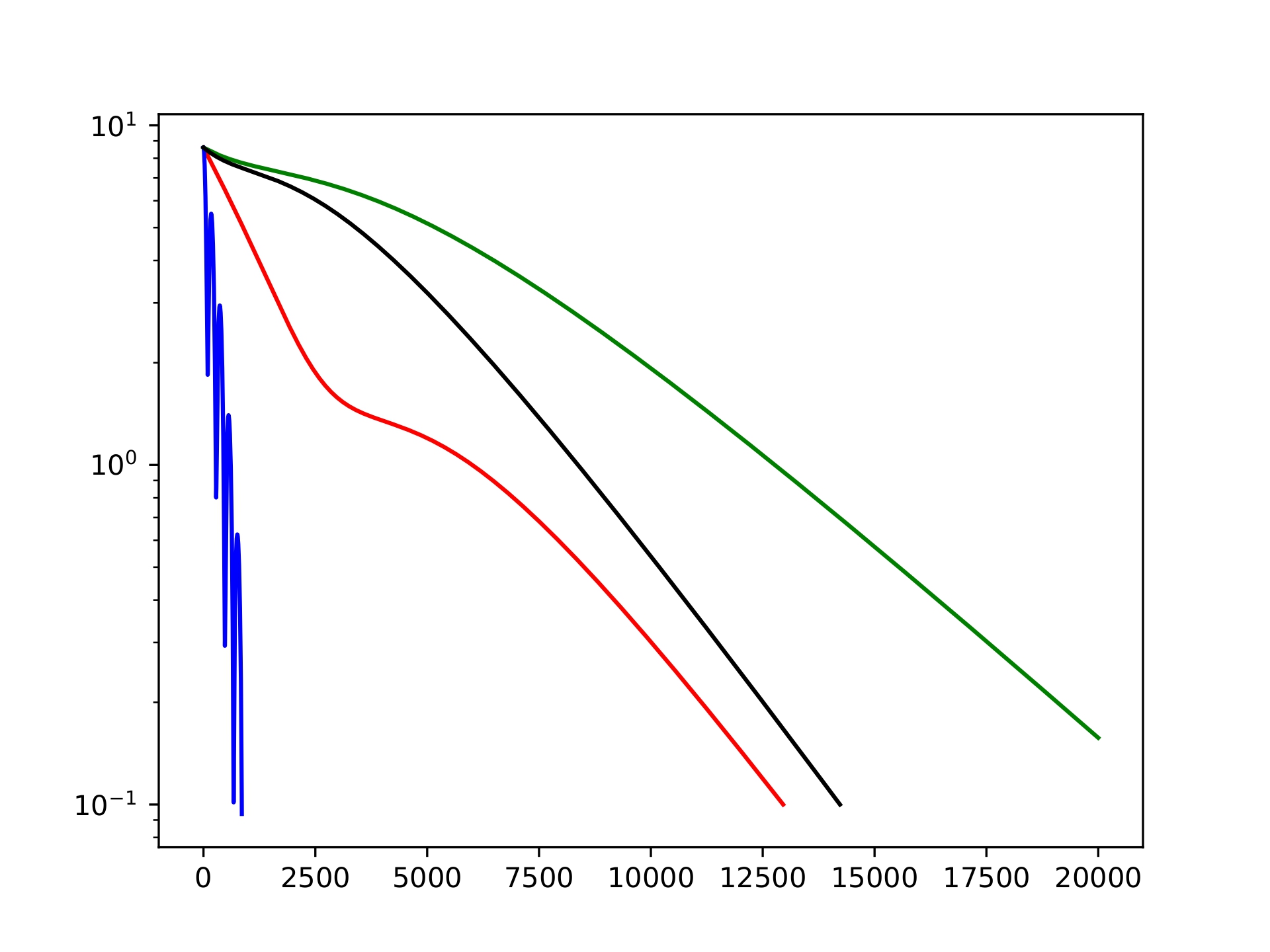}\\ 
\caption{Comparing norm of $x_n$ generated by corresponding parameters of the algorithm \eqref{202403251017}.} 
\label{fig2}
\end{figure}
Besides, in comparison with algorithm \eqref{202403251017} derived from the first order dynamical system, we choose the same parameter $\tau = 0.000146$.
With the error $\epsilon =0.1$, the results of the algorithm \eqref{1stAlg} are demonstrated in Figure \ref{fig3}. As we can see, the algorithm \eqref{202403251017} reaches the error after 12957 steps while the algorithm \eqref{1stAlg} needs 20745 steps, which is significantly higher. The better performance of the algorithm \eqref{202403251017} comes from the inertial effects of the iterations.

\begin{figure}
\centering
\includegraphics[scale=0.7]{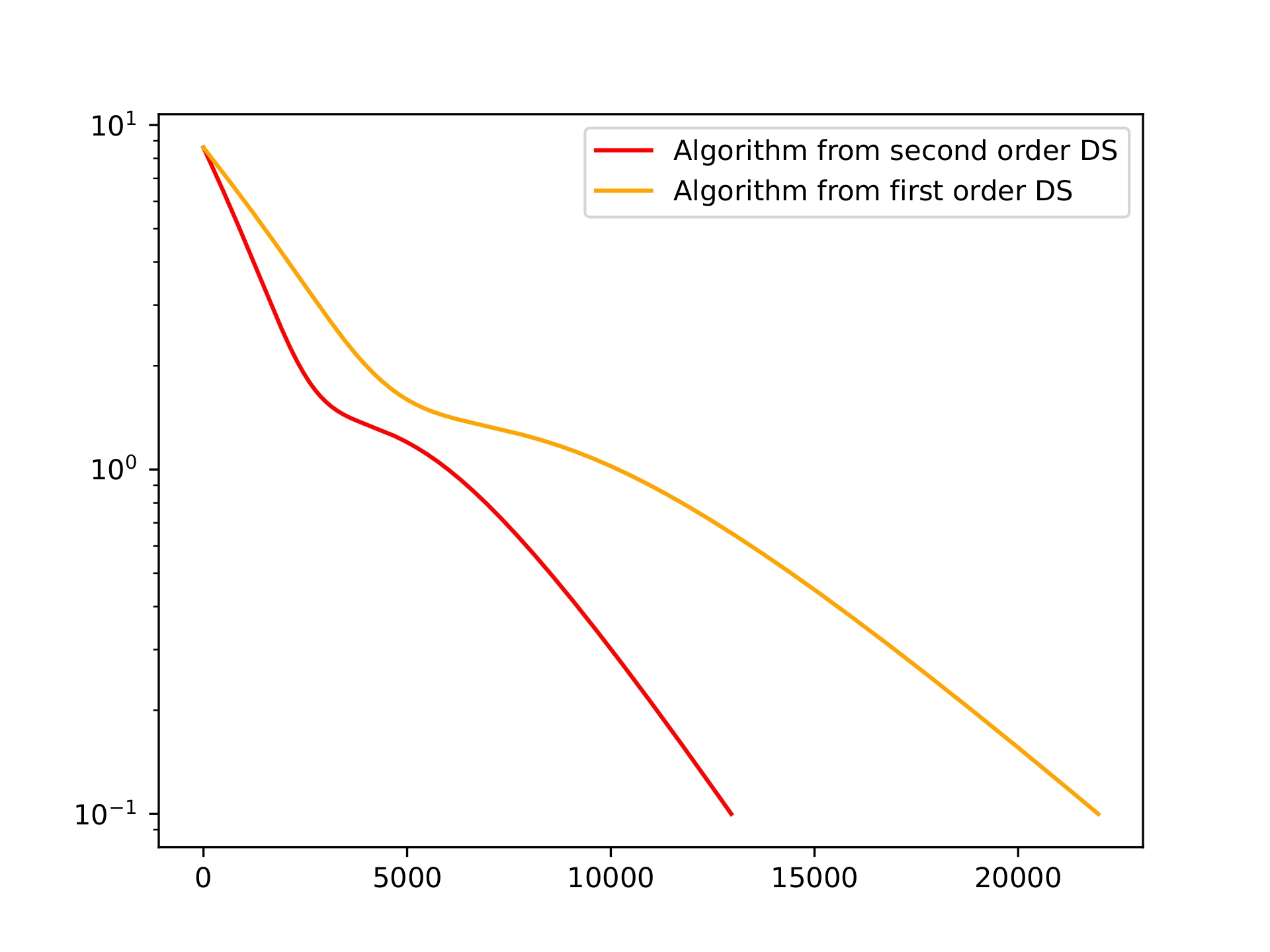}\\ 
\caption{Performance of sequences generated by scheme \eqref{202403251017} and \eqref{1stAlg}.} 
\label{fig3}
\end{figure}

\subsection{Application in Traffic Assignment}
In this part, we recall the traffic assignment problem in \cite{Thanh2024}, which is formed as an IQVIP. We will use algorithm \eqref{202403251017} to solve and compare this algorithm with algorithm \eqref{1stAlg} in term of the residual.\\
\begin{figure}
\centering
\includegraphics[scale=0.8]{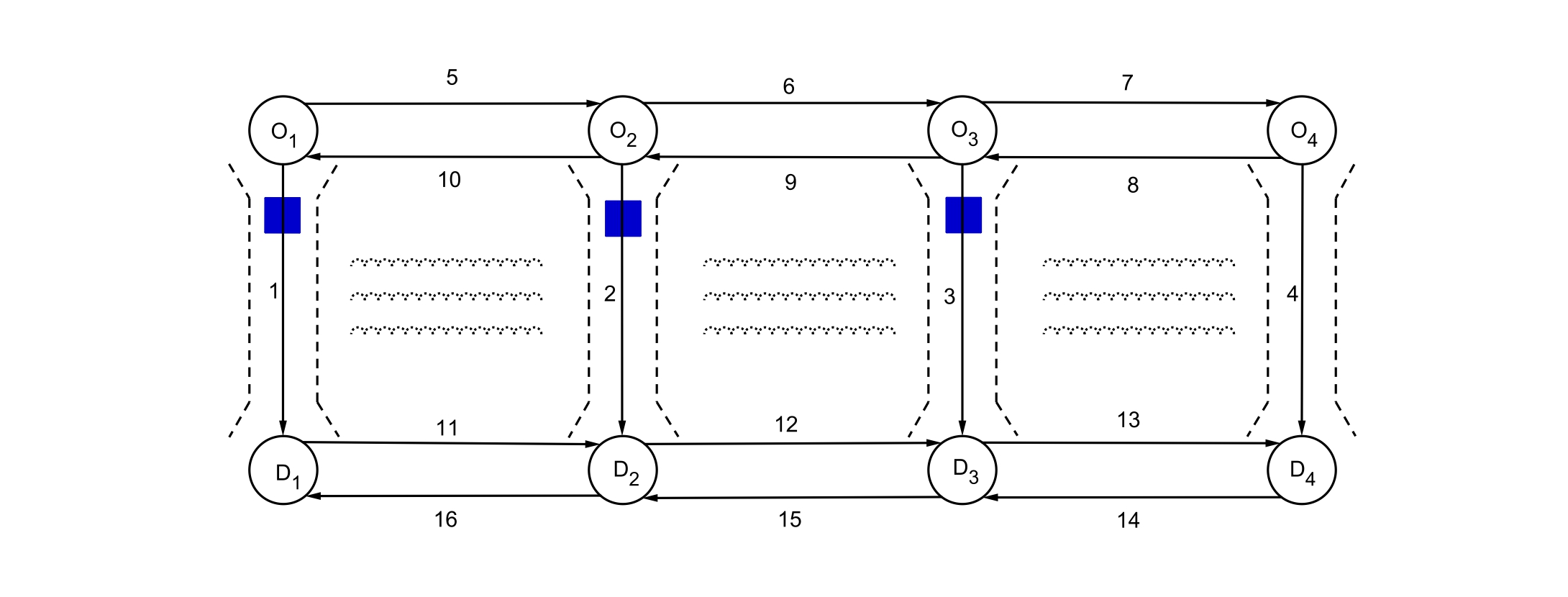}\\ 
\caption{Road pricing problem with four bridge network- Source from \cite{Thanh2024}} 
\label{ODpic}
\end{figure}
We consider the network in which the government attempts to control the link flow $V_i$ by imposing extra tolls $x_i$ on the link $i$ for some special reason. Hence, we consider the link flow $V$ as a function of imposed toll $x$. Specifically, we assume that there are 8 nodes and 16 links connecting these nodes in the traffic network which is presented in Figure \ref{ODpic}. The links 1, 2, 3, 4 are four bridges connecting the origin $O_i$ to the destination $D_i$ ($1\leq i \leq 4 $). Assume further that we already have the citizen's demand for travelling on each link. The policy makers want to control the flow on link 1, 2, 3 to fit with their capacity and flexible conditions by imposing tolls $x_1, x_2, x_3$. For convenient, we employ the detail condition in \cite{Thanh2024} to compare two algorithms. Let $\psi(x)=\psi(x_1,x_2,x_3)=\{(a_1,a_2,a_3)\}$ be the set-valued mapping of conditions satisfying
\[40+x_1\leq a_1\leq 90+x_1; x_2\leq a_2\leq x_2+50; 100+x_3 \leq a_3 \leq 200+x_3. \]
To meet the government's goal, we need the link flow $V(x) \in \psi(x)$. Therefore, from the theoretical analysis in \cite{Thanh2024}, we form an IQVIP derived from this problem as follows.\\
Find the toll $x^*$ satisfied 
\[V(x^*) \in \psi(x^*) \quad \text{and } (z-V(x^*))^T x^* \leq 0,\quad \forall z\in \psi(x^*),\]
which can be rewritten equivalently
\begin{equation}\label{IQVI_TA}
    W(x^*)\in -\psi(x^*) \quad \text{and } (z-W(x^*))^Tx^* \geq 0, \quad \forall z\in -\psi(x^*),
\end{equation}
where $W=-V$.\\
\begin{table}[ht]
\centering
\includegraphics[scale=0.5]{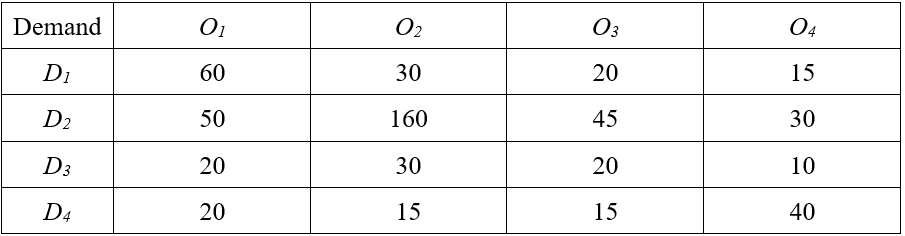}\\ 
\caption{Origin-destination demand table} 
\label{ODdemand}
\end{table}

\begin{table}[ht]
\centering
\includegraphics[scale=0.5]{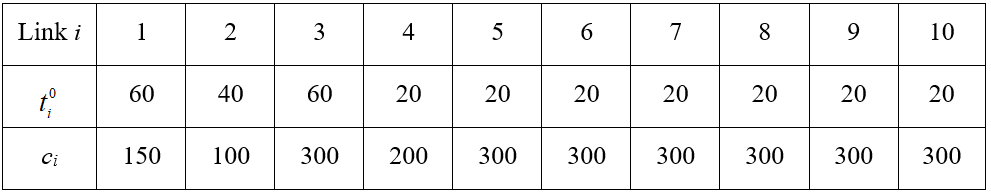}\\ 
\caption{Link free flow travel time and capacity} 
\label{tf}
\end{table}
To calculate the value of link flow at every imposed toll point, we use the user's equilibrium condition for the experiment purpose. As the matter of fact, this data can be collected through observing the number of vehicles on each link in real life. The data about demand, capacity, link travel time between 16 nodes are given in Tables \ref{ODdemand} and \ref{tf}. We use the Bureau of Public Roads (BPR) function to establish the relationship between link travel time and the link flow: 
\[t_i(V_i)= t_i^0 \left[1+0.15 \left(\frac{V_i}{c_i}\right)^4\right],\]
where $V_i,t_i^0$ and $c_i$ denote link flow, free flow travel time and capacity on link $i$, respectively. Now we are ready to use the algorithm \eqref{202403251017} to solve \eqref{IQVI_TA}. Our procedure is presented as follows. We start with the toll $x_0=(0,0,0)$ with the assumption that the toll has not been changed. From the user's equilibrium condition, we figure out the value of link flow at $x_0$, then calculate the next imposed toll by the algorithm \eqref{202403251017} and continue the next iteration. Noting that the solution $x^*$ of \eqref{IQVI_TA} satisfies
 \[V(x) = P_{\psi (x)}(V(x)+\mu x),\]
 so we use the residual term $r_n= \| P_{\psi(x_n)}(V(x_n)+\mu x_n) -V(x_n)\|$ to illustrate the convergence rate of this algorithm. Of course, if the algorithm converges, then the residual converges to zero.\\
 In the first experiment, we take $\sigma = 0.6, \tau = 0.02, \mu=0.5$ . In the second, we take $\sigma=0.6,\tau = 1/30,\mu =0.8$. The performance of algorithm within 150 time step is demonstrated by the link flow and the residual is showed in Figures \ref{fig4} and \ref{fig5}. From two figures, we can see the algorithm \eqref{202403251017} works well where the residual converges linearly to zero. In comparison with the algorithm \eqref{1stAlg}, the algorithm \eqref{202403251017} converges faster in both experiments with the same parameter $\tau$ and $\mu$, and the results are presented in Figure \ref{fig6}.
 
 \begin{figure}
   \raggedleft
  \begin{subfigure}
 \centering
 \includegraphics[width=0.4\linewidth]{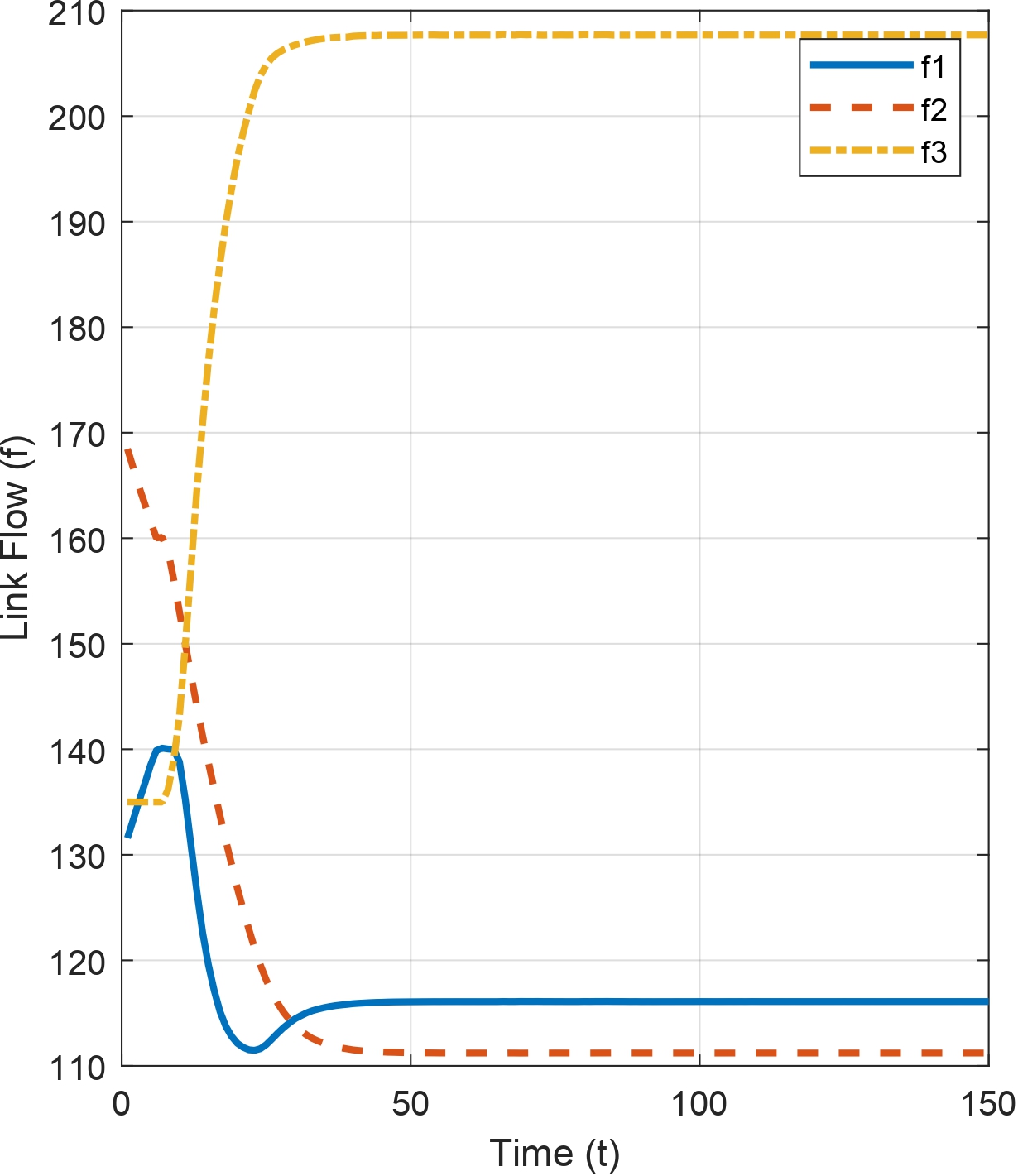}   
  \end{subfigure}%
\hfill
  \begin{subfigure}
  \centering
    \includegraphics[width=0.4\linewidth]{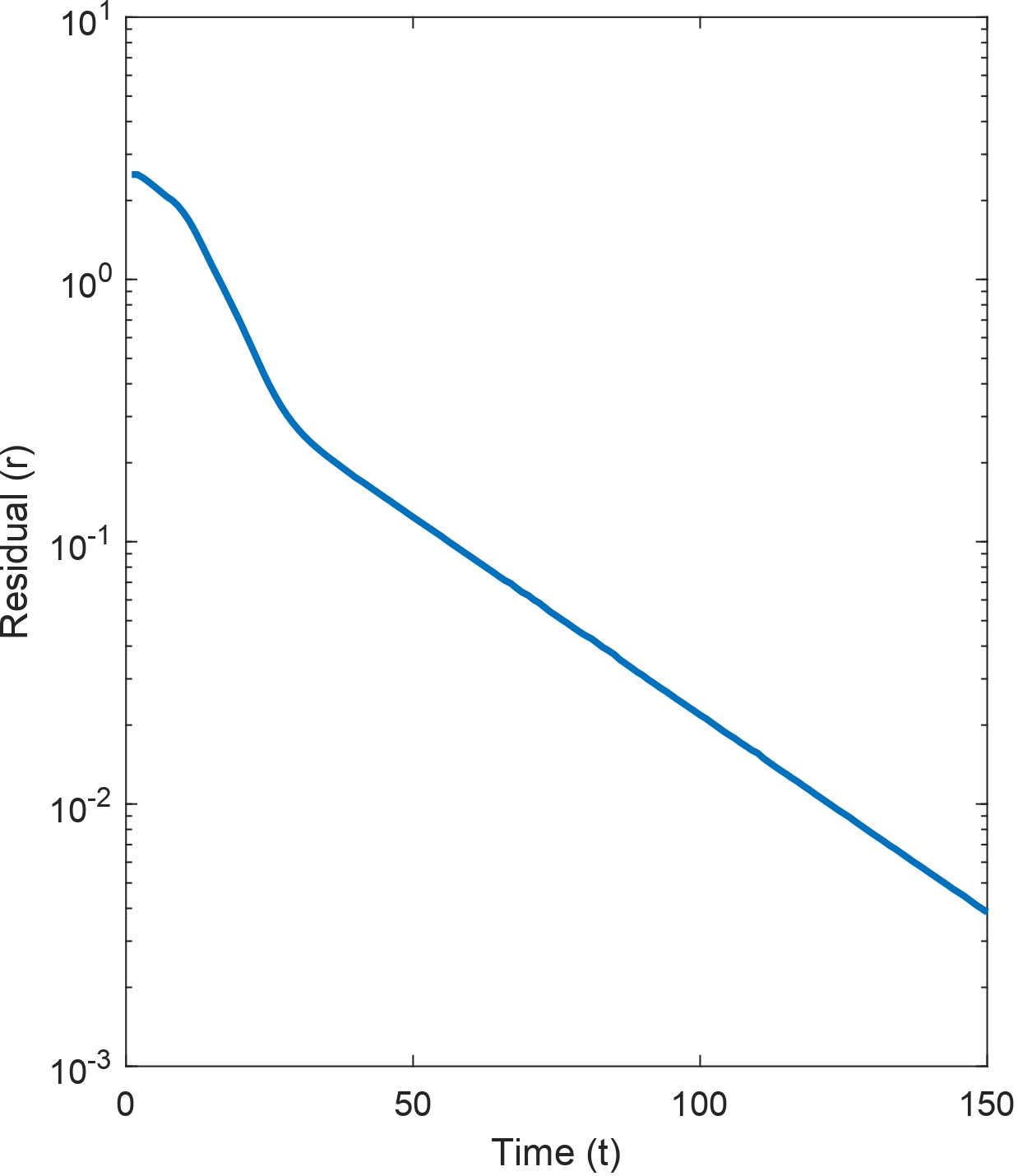}
  \end{subfigure}
  \caption{Link flows and the residual of the projection  algorithm \eqref{202403251017} with $\sigma=0.6, \tau = 0.02$ and $\mu =0.5$.}
  \label{fig4}
\end{figure}

 \begin{figure}
   \raggedleft
  \begin{subfigure}
 \centering
 \includegraphics[width=0.4\linewidth]{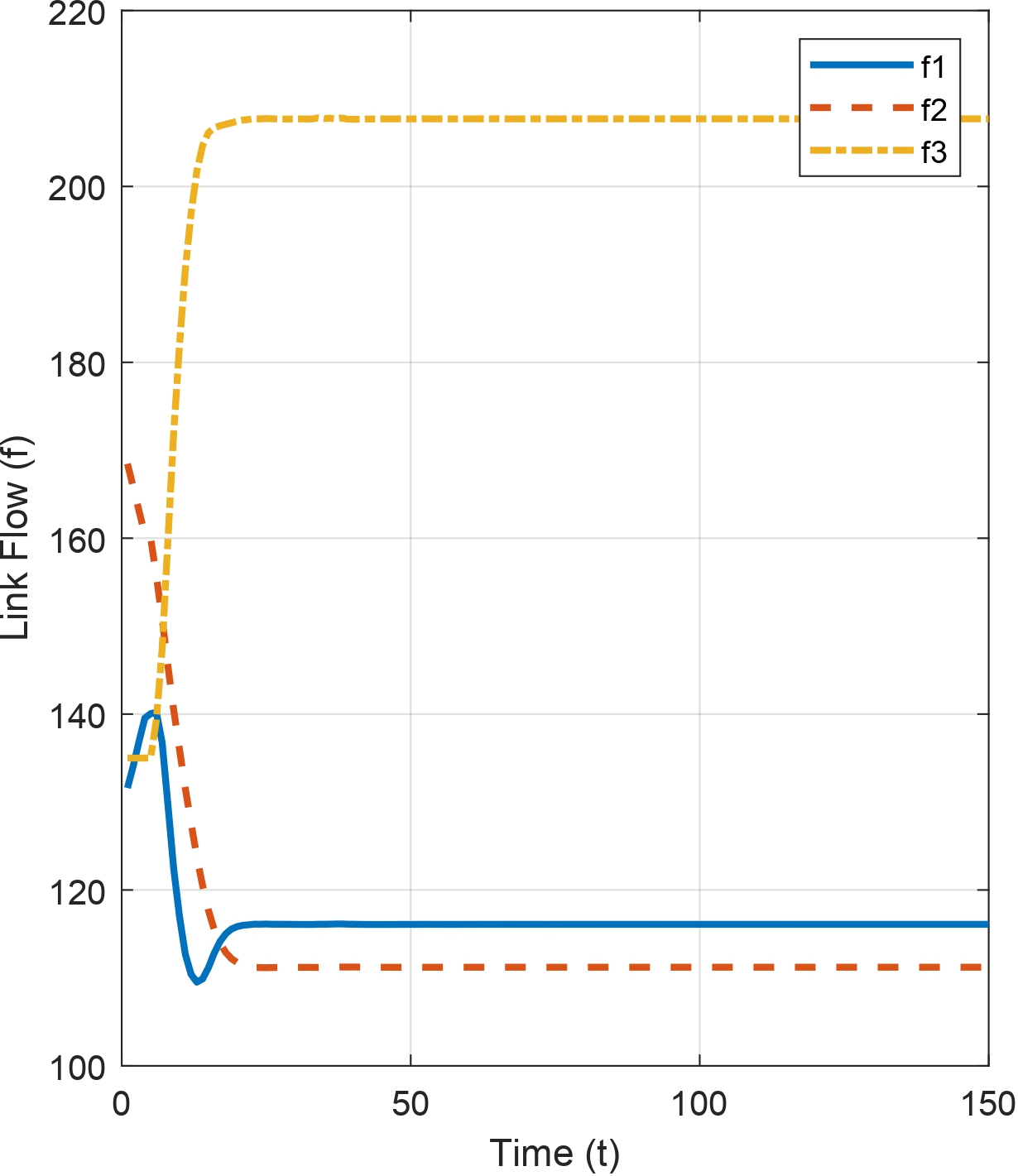}   
  \end{subfigure}%
\hfill
  \begin{subfigure}
  \centering
    \includegraphics[width=0.4\linewidth]{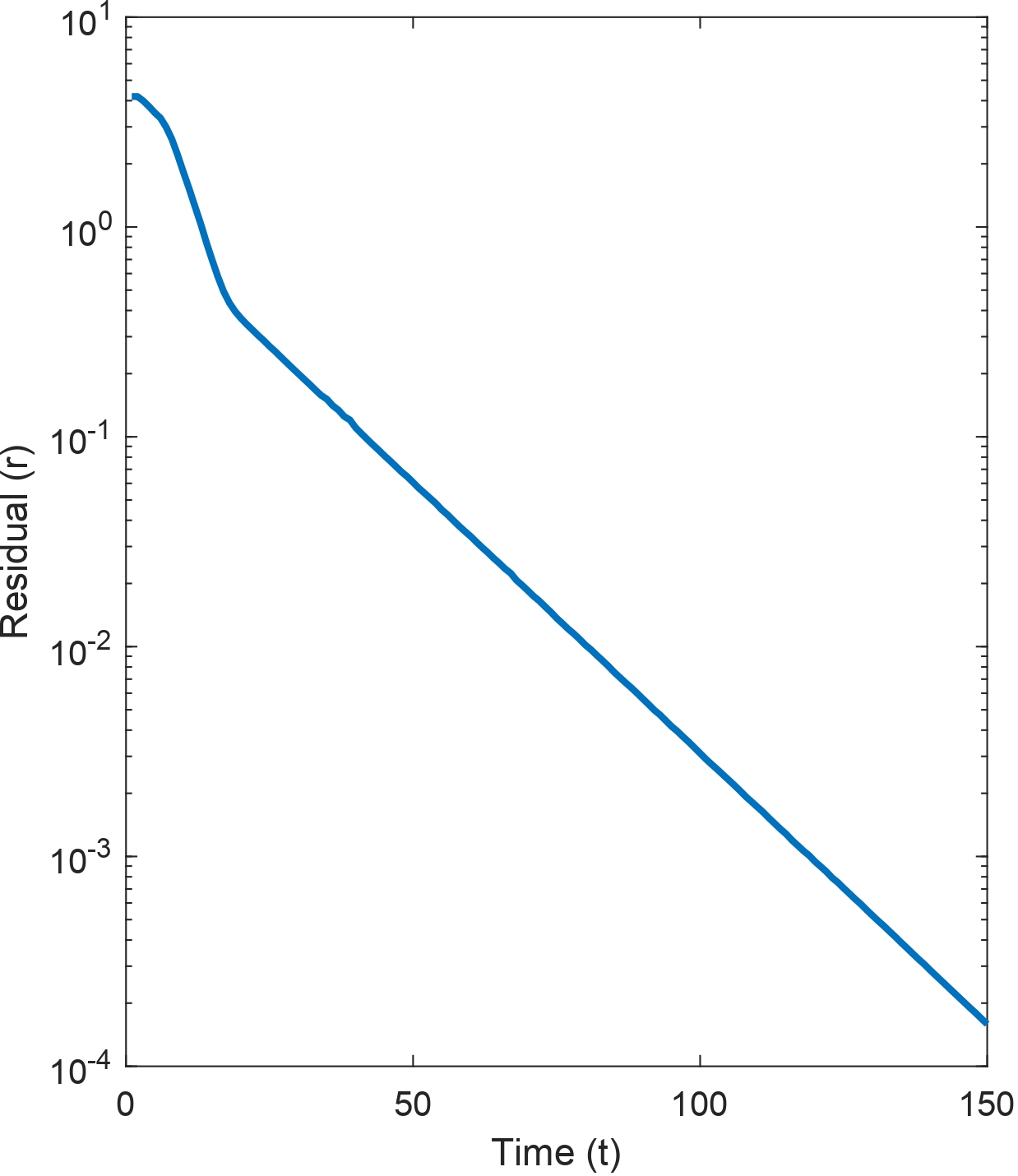}
  \end{subfigure}
  \caption{Link flows and the residual of the projection  algorithm \eqref{202403251017} with $\sigma=0.6, \tau = 1/30$ and $\mu =0.8$.}
  \label{fig5}
\end{figure}

\begin{figure}
   \raggedleft
  \begin{subfigure}
 \centering
 \includegraphics[width=0.4\linewidth]{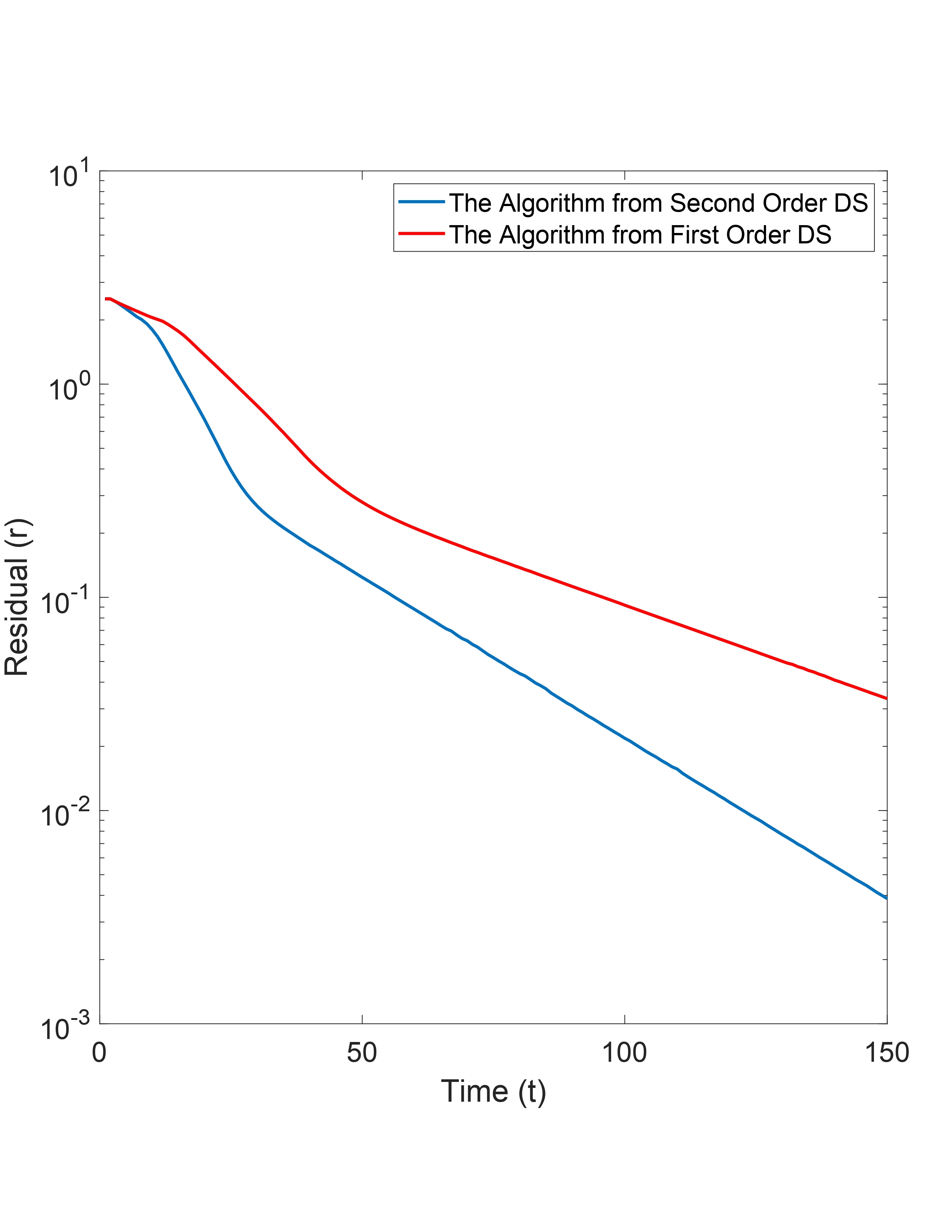} 
  \end{subfigure}%
\hfill
  \begin{subfigure}
  \centering
    \includegraphics[width=0.4\linewidth]{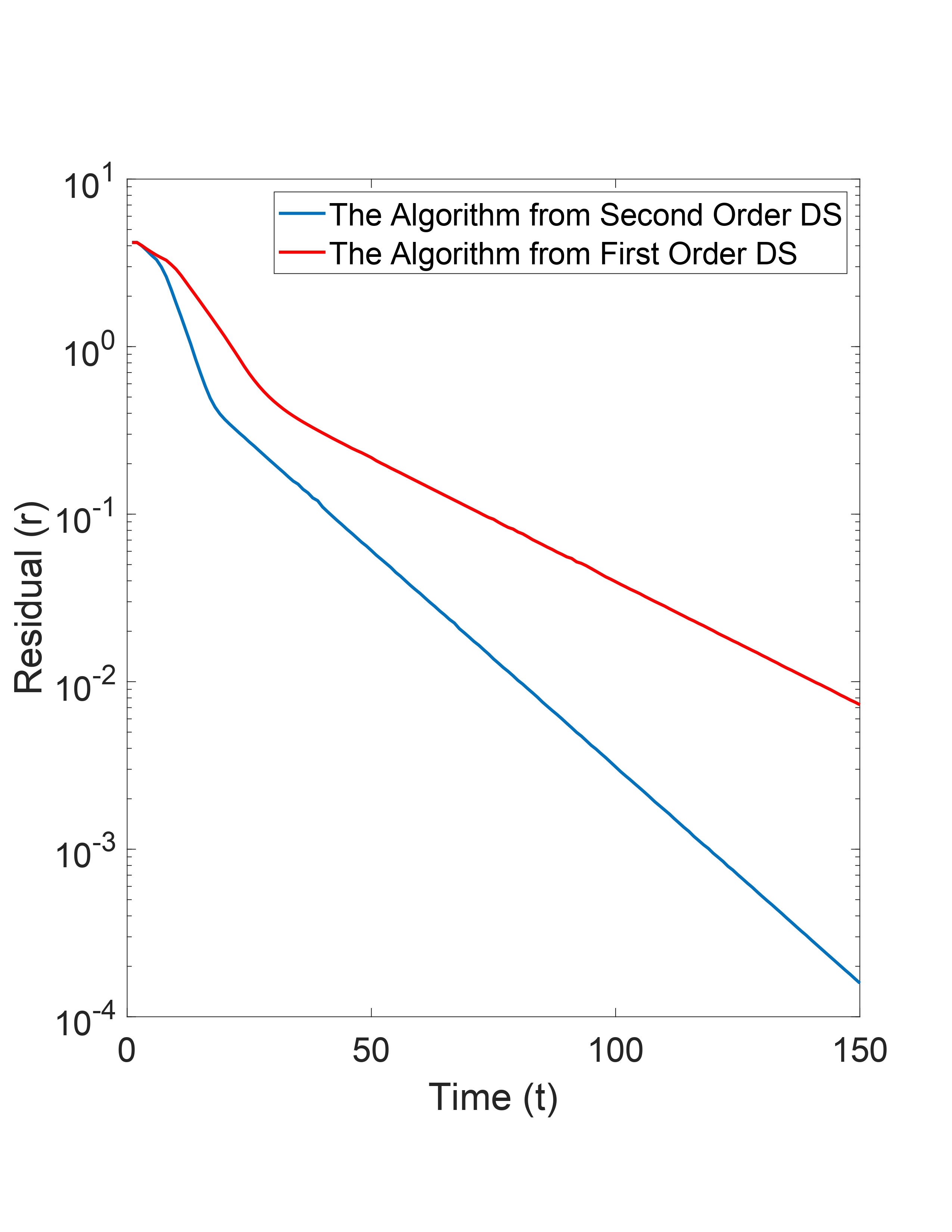}
  \end{subfigure}
  \caption{Residual of the algorithm \eqref{202403251017} and \eqref{1stAlg} with same parameter $\tau$ and $\mu$.}
  \label{fig6}
\end{figure}

\newpage 

\section*{Declarations}
{\bf Conflict of interest} The authors declare no competing interests.

\end{document}